\documentclass{amsart}

\usepackage{graphicx}
\usepackage{tabularx,ragged2e}
\newcolumntype{C}{>{\Centering\arraybackslash}X} % centered "X" column\usepackage{tikz-cd}
\usepackage[all,cmtip]{xy}
\usepackage{ amssymb }
\usepackage[margin=1.2 in]{geometry}
\usepackage[utf8]{inputenc}

\usepackage{amssymb}
\usepackage{amsmath}
\usepackage{amsthm}
\usepackage{amsfonts}
\usepackage{stmaryrd}
\usepackage{dsfont}
\usepackage{xfrac}
\usepackage{mathrsfs}
\usepackage{hyperref}
\usepackage{todonotes}

\newtheorem{theorem}{Theorem}[section]
\newtheorem{lemma}[theorem]{Lemma}

\theoremstyle{definition}
\newtheorem{definition}[theorem]{Definition}

\theoremstyle{remark}
\newtheorem{remark}[theorem]{Remark}

\numberwithin{equation}{section}

\DeclareMathOperator{\Id}{Id}
\DeclareMathOperator{\BN}{BN}
\DeclareMathOperator{\Kh}{Kh}

\newcommand{\A}{\alpha}
\usepackage{tikz}
\newcommand*\circled[1]{\tikz[baseline=(char.base)]{
    \node[shape=circle, draw, inner sep=1pt, 
        minimum height=12pt] (char) {#1};}}

\usepackage{tikz} \usetikzlibrary{matrix,arrows,shapes,calc}
\tikzstyle{knot}=[thick]
\tikzstyle{knott}=[thick,preaction={draw, line width=4pt, white}]
\tikzstyle{crossing}=[circle,fill=white,inner sep=0,outer sep=0,minimum width=3pt]
\tikzstyle{dot}=[circle,fill=black,inner sep=0,outer sep=0,minimum width=4pt]

\newcommand{\wt}{\widetilde}

\usepackage{amsmath, amsthm, amsfonts, amssymb, enumerate, wrapfig,
tikz,  color, rotating, colortbl,   datetime,  float,  stmaryrd,
hyperref, bookmark, cleveref, ifthen, tikz-3dplot, setspace, mathtools}
\usetikzlibrary{decorations.markings, decorations.pathreplacing, decorations.pathmorphing}
\usepackage{amscd}

\usetikzlibrary{arrows, decorations.markings, decorations, patterns, positioning, decorations.pathreplacing, decorations.pathmorphing}
\tikzset{->-/.style={decoration={markings, mark=at position .5 with {\arrow{>}}},postaction={decorate}}}
\tikzset{-<-/.style={decoration={markings, mark=at position .5 with {\arrow{<}}},postaction={decorate}}}

\begin{document}

	\title{Ribbon distance bounds from Bar-Natan Homology and $\A$-Homology}
	\author{Onkar Singh Gujral}
\address{Kolkata, West Bengal, India}
\email{osgujral@gmail.com}
	\maketitle

\begin{abstract}
We prove a lower bound on the ribbon distance via Bar-Natan Homology. This lower bound agrees with Alishahi's lower bound on the unknotting number via Bar-Natan Homology, which furthers a pattern first observed by Sarkar: Sarkar's lower bound on the ribbon distance via Lee Homology agreed with Alishahi and Dowlin's lower bound on the unknotting number via Lee Homology. We also prove a lower bound on the ribbon distance via $\A$-homology, a variant of Khovanov homology defined recently by Khovanov and Robert.
\end{abstract}

%BEGINNING OF AFTER BEGIN DOC
\let\oldemptyset\emptyset
\let\emptyset\varnothing

\newcounter{res}[section]
\numberwithin{res}{section}
\newtheorem{thm}[res]{Theorem}
\newtheorem*{theo}{Theorem}
\newtheorem*{ques}{Question}
\newtheorem*{claim}{Claim}
\newtheorem{lem}[res]{Lemma}
\newtheorem{obs}[res]{Observation}
\newtheorem{prop}[res]{Proposition}
\newtheorem{cor}[res]{Corollary}
\newtheorem{que}[res]{Question}
\theoremstyle{definition}
\newtheorem{notation}[res]{Notation}
\newtheorem{dfn}[res]{Definition}
\newtheorem{rmk}[res]{Remark}
\newtheorem{exa}[res]{Example}
\newtheorem{exo}[res]{Exercise}
\newtheorem{cjc}[res]{Conjecture}
\newcommand{\imagesfolder}{.}
\setlength{\marginparwidth}{1.2in} \let\oldmarginpar\marginpar
\newcommand\marginparL[1]{\oldmarginpar{\color{red}\fbox{\begin{minipage}{3cm}
  \footnotesize #1 \end{minipage}}}}
\newcommand\marginparMK[1]{\oldmarginpar{\color{blue}\fbox{\begin{minipage}{3cm}
  \footnotesize #1 \end{minipage}}}}

\def\co{\colon\thinspace}
\newcommand{\NB}[1]{\ensuremath{\vcenter{\hbox{#1}}}}
\newcommand{\ttkz}[2][]{\NB{\tikz[#1]{\input{\imagesfolder/#2}}}}
\newcommand{\NN}{\ensuremath{\mathbb{N}}}
\newcommand{\FF}{\ensuremath{\mathbb{F}}}
\newcommand{\ZZ}{\ensuremath{\mathbb{Z}}}
\newcommand{\CC}{\ensuremath{\mathbb{C}}}
\newcommand{\QQ}{\ensuremath{\mathbb{Q}}}
\newcommand{\RR}{\ensuremath{\mathbb{R}}}
\newcommand{\DD}{\ensuremath{\mathbb{D}}}
\newcommand{\KK}{\ensuremath{\mathbb{K}}}
\newcommand{\PP}{\ensuremath{\mathbb{P}}}
\renewcommand{\SS}{\ensuremath{\mathbb{S}}}
\renewcommand {\Im}{\operatorname{Im}} %\newcommand
%{\Id}{\operatorname{Id}} \newcommand{\id}{\mathrm{Id}}
\newcommand{\tr}{\mathop{\mathrm{Tr}}\nolimits}
\newcommand{\Ker}{\mathop{\mathrm{Ker}}\nolimits}
\newcommand{\Hom}{\mathop{\mathrm{Hom}}}
\newcommand{\End}{\mathop{\mathrm{End}}}
\newcommand{\qdim}{\mathop{\mathrm{dim}_q}\nolimits}
\newcommand{\sdim}{\mathop{\mathrm{sdim}}\nolimits}
\newcommand{\gll}{\ensuremath{\mathfrak{gl}}}
\newcommand{\sll}{\ensuremath{\mathfrak{sl}}}
\newcommand{\rk}{\mathrm{rk}}
\newcommand\eqdef{\ensuremath{\stackrel{\textrm{def}}{=}}}

\newcommand\kup[1]{\left\langle #1 \right\rangle}
\newcommand\kupo[1]{\left\langle #1 \right\rangle_{\omega}}
\newcommand\kups[1]{\left\llangle #1 \right\rrangle}

\newcommand{\QvectZ}{\ensuremath{\QQ\textrm{-}\mathsf{vect}_{\ZZ\textrm{-}\mathrm{gr}}}}
\newcommand{\vectgr}{\ensuremath{\mathsf{vect}_{\mathrm{gr}}}}
\newcommand{\svect}{\ensuremath{\mathsf{Svect}}}
\newcommand{\vect}{\ensuremath{\mathsf{vect}}}
\newcommand{\svectgr}{\ensuremath{\mathsf{Svect}_{\mathrm{gr}}}}
\newcommand{\Foam}{\ensuremath{\mathsf{Foam}}}
\newcommand{\qvg}{\ensuremath{\QQ\mathsf{-vect}_{\mathrm{gr}}}}
\newcommand{\blmoy}{\ensuremath{\mathsf{blMOY}}}
\newcommand{\Qalggr}{\ensuremath{\QQ\mathrm{-}\mathsf{Alg}_{\mathrm{gr}}}}

\newcommand{\longto}{\ensuremath{\longrightarrow}} 

\newcommand{\qbina}[2]{\ensuremath
\begin{bmatrix}
  #1 \\
  #2
\end{bmatrix}
} \newcommand{\qbinil}[2]{\ensuremath \left[\begin{smallmatrix}
    #1 \\
    #2
\end{smallmatrix} \right]
} \newcommand{\qbin}[2]{\ensuremath
\begin{bmatrix}
  #1 + #2 \\
  #1 \quad #2
\end{bmatrix}
} \newcommand{\qbinb}[3]{\ensuremath
\begin{bmatrix}
  #1 \\
  #2 \quad #3
\end{bmatrix}
} \newcommand{\trinomial}[4]{\ensuremath{
\begin{pmatrix}
      #1 \\ #2\quad #3 \quad #4
   \end{pmatrix}
    }}

\newcommand{\Xing}{\NB{\scalebox{1.3}{\ensuremath{\times}}}}

\newcommand{\col}[2][{}]{\ensuremath{\mathrm{col}_{#1}(#2)}} 
\newcommand{\R}{\ensuremath{R}} 
\newcommand{\Tait}{\mathrm{Tait}}
\newcommand{\cone}{\mathrm{Cone}}
\newcommand{\Z}{\mathbb{Z}}
\newcommand{\Gd}{\ensuremath{G_d}}
\newcommand{\lra}{\longrightarrow}
\newcommand{\kk}{\mathbf{k}}
\newcommand{\mcD}{\mathcal{D}}
\newcommand{\mcF}{\mathcal{F}}
\newcommand{\mcZ}{\mathcal{Z}}
\newcommand{\Fg}{\mathsf{Fg}}
\newcommand{\adm}{\mathrm{adm}}
\newcommand{\Cob}{\mathsf{Cob}}
\newcommand{\undeps}{\underline{\epsilon}}
\newcommand{\undell}{\underline{\ell}}

\newcommand{\mtH}{\mathrm{H}}
\newcommand{\Smod}{\ensuremath{S\textrm{-}\mathsf{mod}}}
\newcommand{\leftsquigarrow}{\ensuremath{\rotatebox[origin=c]{180}{\NB{$\rightsquigarrow$}}}}

\newcommand{\SeSu}{\ensuremath{\mathsf{SeSu}}}
\newcommand{\BSeSu}{\ensuremath{\mathsf{BSeSu}}}
\newcommand{\CSeSu}{\ensuremath{\mathsf{CSeSu}}}
\newcommand{\CCSeSu}{\ensuremath{\mathsf{Cob}^{1+1}_{\RR^3}}}

\newcommand{\cfCircleSplit}[3]{\NB{\tikz[scale=#1]{
\begin{scope}
  \begin{scope}[yshift= 1.5cm]
    \draw (1,0) arc (0:360:1cm and 0.3cm);
    \draw[thin] (1,0) arc (0:-180:1cm and 1.2cm);
    \node at (0, -0.75) {$#2$};
  \end{scope}
  \begin{scope}[yshift =-1.5cm]
    \draw (1,0) arc (0:-180:1cm and 0.3cm); 
    \draw[densely dotted] (1,0) arc (0:180:1cm and 0.3cm); 
    \draw (1,0) arc (0:180:1cm and 1.2cm);
    \node at (0, 0.75) {$#3$};
  \end{scope}
\end{scope}}}}

\newcommand{\cfCircleId}[2][]{\NB{\tikz[scale=#2]{
\begin{scope}
  \begin{scope}[yshift= 1.5cm]
    \draw (1,0) arc (0:360:1cm and 0.3cm)coordinate [pos=0](e) coordinate [pos=0.5](f);
  \end{scope}
  \node at (0,0) {$#1$};
  \begin{scope}[yshift =-1.5cm]
    \draw (1,0) arc (0:-180:1cm and 0.3cm)coordinate [pos=0](g) coordinate [pos=1](h);
    \draw[densely dotted] (1,0) arc (0:180:1cm and 0.3cm); 
  \end{scope}
  \draw (e) -- (g);
  \draw (f) -- (h); 
\end{scope}}}}

\newcommand{\cfDecSquare}[2]{\NB{\tikz[scale=#1]{
\begin{scope}
  \node at (0,0) {$#2$};
  \draw (-1, -1) -- (-1,1) -- (1,1)  -- (1, -1) -- (-1,-1);
\end{scope}}}}

\newcommand{\cfDoubleSeamA}[1]{\NB{\tikz[scale=#1]{
\begin{scope}[decoration={border,segment length=1mm,amplitude=#1mm,angle=90}]
  \draw (-1, -1) -- (-1,1) -- (1,1)  -- (1, -1) -- (-1,-1);
  \draw[red, postaction={draw, decorate}] (-1, 0) .. controls +
  (0.5,0) and  +(0,0.5) .. (0, -1);
  \draw[red, postaction={draw, decorate}] ( 1, 0) .. controls +
  (-0.5,0) and  +(0,-0.5) .. (0,  1);
\end{scope}}}}

\newcommand{\cfDoubleSeamB}[1]{\NB{\tikz[scale=#1]{
\begin{scope}[decoration={border,segment length=1mm,amplitude=#1mm,angle=90}]
  \draw (-1, -1) -- (-1,1) -- (1,1)  -- (1, -1) -- (-1,-1);
  \draw[red, postaction={draw, decorate}] (1, 0) .. controls +
  (-0.5,0) and  +(0,0.5) .. (0, -1);
  \draw[red, postaction={draw, decorate}] (-1, 0) .. controls +
  (0.5,0) and  +(0,-0.5)  .. (0,  1);
\end{scope}}}}

\newcommand{\cfSeamSquareTwo}[3]{\NB{\tikz[rotate= -90, scale=#1]{
\begin{scope}[decoration={border,segment length=1mm,amplitude=#1mm,angle=90}]
  \node at (-.5,0) {$#2$};
  \node at (0.5,0) {$#3$};
  \draw (-1, -1) -- (-1,1) -- (1,1)  -- (1, -1) -- (-1,-1);
  \draw[red, postaction={draw, decorate}] (0,-1) -- (0, 1);
\end{scope}}}}

\newcommand{\cfSeamSquare}[3]{\NB{\tikz[scale=#1]{
\begin{scope}[decoration={border,segment length=1mm,amplitude=#1mm,angle=90}]
  \node at (-.5,0) {$#2$};
  \node at (0.5,0) {$#3$};
  \draw (-1, -1) -- (-1,1) -- (1,1)  -- (1, -1) -- (-1,-1);
  \draw[red, postaction={draw, decorate}] (0,1) -- (0, -1);
\end{scope}}}}

\newcommand{\cfGenusTwoBdy}[1]{\NB{\tikz[scale=#1]{
\begin{scope}
  \draw (0,0) circle (1cm and 0.3cm);
  \draw[very thin] (-1, 0)
  .. controls +(0,-0.3) and +(0,1) .. (-3, -2)
  .. controls +(0,-2) and +(-1,0) .. (0, -2)
  .. controls +(1, 0) and +(0,-2) .. (3, -2)
  .. controls +(0, 1) and +(0,-0.3) .. (1, 0);
  \draw[very thin] (-1.1, -1) arc (340: 280: 1.5) coordinate [pos =0.9]
  (a);
  \draw[very thin] (a) arc (154: 106: 1.5);
  \draw[very thin] (1.1, -1) arc (200: 260: 1.5) coordinate [pos =0.9]
  (a);
  \draw[very thin] (a) arc ( 26: 74: 1.5);
\end{scope}
}}}

\newcommand{\cfGenusTwoCup}[2]{\NB{\tikz[scale=#1]{
\begin{scope}
  \draw (0,0) circle (1cm and 0.3cm);
  \draw[very thin] (1,0) arc (360:180:1cm and 1cm);
  \node at (0, -2) {$#2$}; 
  \draw[very thin](-3, -2)
  .. controls +(0,-2) and +(-1,0) .. (0, -2.5)
  .. controls +(1, 0) and +(0,-2) .. (3, -2)
  .. controls +(0,2) and +(1,0) .. (0, -1.5)
  .. controls +(-1, 0) and +(0,2) .. (-3, -2);
  \draw[very thin] (-1, -2) arc (300: 240:1.5) coordinate[pos=0.85]
  (a);
  \draw[very thin] (a) arc (111: 69:1.5);
  \draw[very thin] ( 1, -2) arc (240: 300:1.5) coordinate[pos=0.85]
  (b);
    \draw[very thin] (b) arc (69: 111:1.5);
\end{scope}
}}}

\newcommand{\cfCup}[2]{\NB{\tikz[scale=#1]{
\begin{scope}
  \draw (0,0) circle (1cm and 0.3cm);
  \draw[very thin] (1,0) arc (360:180:1cm and 1.2cm);  
  \node at (0, -0.75) {$#2$};
\end{scope}
}}}

\newcommand{\cfCap}[2]{\NB{\tikz[scale=#1]{
\begin{scope}
  \draw (0,1) arc (0:-180: 1cm and 0.3cm);
  \draw[densely dotted] (0,1) arc (0:180: 1cm and 0.3cm);
  \draw[very thin] (0,1) arc (0:180:1cm and 1.2cm);  
  \node at (0, 0.75) {$#2$};
\end{scope}
}}}

\newcommand{\cfSphere}[2]{\NB{\tikz[scale=#1]{
\begin{scope}
  \draw[very thin] (1,0) arc (0:-180: 1cm and 0.3cm);
  \draw[very thin, densely dotted] (1,0) arc (0:180: 1cm and 0.3cm);
  \draw[very thin] (1,0) arc (360:0:1cm and 1cm);  
  \node at (0, -0.65) {$#2$};
\end{scope}
}}}

\newcommand{\cfSphereSpecial}[3]{\NB{\tikz[scale=#1]{
\begin{scope}
  \draw[very thin] (1,0) arc (0:-180: 1cm and 0.3cm);
  \draw[very thin, densely dotted] (1,0) arc (0:180: 1cm and 0.3cm);
  \draw[very thin] (1,0) arc (360:0:1cm and 1cm);  
  \node at (0, -0.65) {$#2$};
  \node at (0, 0.65) {$#3$};
\end{scope}
}}}

\newcommand{\cfGenusOneBdy}[1]{\NB{\tikz[scale=#1]{
\begin{scope}
  \draw (0,0) circle (1cm and 0.3cm);
  \draw[very thin] (-1, 0)
  .. controls +(0,-0.3) and +(0,1) .. (-2, -2)
  .. controls +(0,-1.5) and +(0,-1.5) .. (2, -2)
  .. controls +(0, 1) and +(0,-0.3) .. (1, 0);
  \draw[very thin] (0, -1.8) arc (270: 300: 1.8) coordinate [pos =0.8]
  (a);
  \draw[very thin] (0, -1.8) arc (270: 240: 1.8);
  \draw[very thin] (a) arc (66: 114: 1.8);
\end{scope}
}}}

\newcommand{\cfDelta}[2]{\NB{\tikz[scale=#1]{
  \begin{scope}
    \draw[densely dotted] (1, -1) arc (0:180:1cm and 0.3cm);
    \draw (1, -1) arc (0:-180:1cm and 0.3cm);
    \draw (-1, 1) arc (0:360:1cm and 0.3cm);
    \draw (3, 1) arc (0:360:1cm and 0.3cm);
    \draw[very thin] (1,  1) .. controls +(0,-0.5) and + (0, -0.5) .. +(-2,0) node [midway, below] {$#2$};
    \draw[very thin] (-1,  -1) .. controls +(0,0.3) and + (0, -0.3) .. +(-2,2);
    \draw[very thin] (1,  -1) .. controls +(0,0.3) and + (0, -0.3) .. +(2,2);
  \end{scope}
}}}

\newcommand{\cfMu}[2]{\NB{\tikz[scale=#1]{
  \begin{scope}
    \draw[densely dotted] (3, -1) arc (0:180:1cm and 0.3cm);
    \draw (3, -1) arc (0:-180:1cm and 0.3cm);
    \draw[densely dotted] (-1, -1) arc (0:180:1cm and 0.3cm);
    \draw (-1, -1) arc (0:-180:1cm and 0.3cm);
    \draw (1, 1) arc (0:360:1cm and 0.3cm);
    \draw[very thin] (1,  -1) .. controls +(0,0.5) and + (0, 0.5) .. +(-2,0) node [midway, below] {$#2$};
    \draw[very thin] (-3,  -1) .. controls +(0,0.3) and + (0, -0.3) .. +(2,2);
    \draw[very thin] (3,  -1) .. controls +(0,0.3) and + (0, -0.3) .. +(-2,2);
  \end{scope}
}}}

\newcommand{\cfGenusThree}[1]{\NB{\tikz[scale=#1]{
\begin{scope}
  \draw (270:1) .. controls +(0:1) and +(240:1)  .. (330:3);
  \draw (30:1) .. controls +(-60:1) and +(60:1)  .. (330:3);
  \draw (333:1)  arc (210:270:1.3cm) coordinate[pos=0.9] (a);
  \draw (a) arc (36: 84: 1.3cm);
\end{scope}
\begin{scope}[rotate=120]
  \draw (270:1) .. controls +(0:1) and +(240:1)  .. (330:3);
  \draw (30:1) .. controls +(-60:1) and +(60:1)  .. (330:3);
  \draw (333:1)  arc (210:270:1.3cm) coordinate[pos=0.9] (a);
  \draw (a) arc (36: 84: 1.3cm);
\end{scope}
\begin{scope}[rotate=-120]
  \draw (270:1) .. controls +(0:1) and +(240:1)  .. (330:3);
  \draw (30:1) .. controls +(-60:1) and +(60:1)  .. (330:3);
  \draw (333:1)  arc (210:270:1.3cm) coordinate[pos=0.9] (a);
  \draw (a) arc (36: 84: 1.3cm);
\end{scope}}}}

\newcommand{\cfTubeTT}[2]{\NB{\tikz[scale=#1]{
  \begin{scope}
    \draw (-1, 1) arc (0:360:1cm and 0.3cm);
    \draw (3, 1) arc (0:360:1cm and 0.3cm);
    \draw[very thin] (1,  1) .. controls +(0,-0.7) and + (0, -0.7) .. +(-2,0) node [midway, below] {$#2$};
    \draw[very thin] (-3,  1) .. controls +(0,-2.1) and + (0, -2.1) .. (3,1);
  \end{scope}
}}}

\newcommand{\circlein}[1]{\NB{\tikz[scale=#1]{
  \begin{scope}[scale =#1, decoration={border,segment length=1mm,amplitude=#1mm,angle=90}]
    \draw[red, postaction={draw, decorate}] (0,0) circle (0.6cm);
  \end{scope}
  }}}

\newcommand{\circleout}[1]{\NB{\tikz[scale=#1]{
  \begin{scope}[scale =#1, decoration={border,segment length=1mm,amplitude=#1mm,angle=-90}]
    \draw[red, postaction={draw, decorate}] (0,0) circle (0.6cm);
  \end{scope}
  }}}

\newcommand{\warning}[1]{\NB{\tikz[scale=#1]{
\node[scale=#1, font=\normalsize] at (0,0.05) {$\mathbf{!}$};
\draw (-30:0.35) --(90:0.35) -- (210:0.35) -- cycle;
}}}

\newcommand{\feSeamedCup}[2]{\NB{\tikz[scale=#1]{
\begin{scope}[decoration={border,segment length=#1mm,amplitude=#1mm,angle=-90}]
  \begin{scope}
    \draw (1,0) arc (0:360:1cm and 0.3cm)coordinate [pos=0](e)
    coordinate [pos=0.5](f) coordinate[pos=0.65] (bt)
    coordinate[pos=0.85] (at);
    \draw[thin] (1,0) arc (0:-180:1cm and 1.2cm);
    \node at (0,-0.5) {$#2$};
  \end{scope}
  \draw[red, postaction={draw, decorate}] (at) .. controls +(0, -1)  and +(0, -1) .. (bt);
\end{scope}
}}}

\newcommand{\feSeamedTwistedTube}[3]{\NB{\tikz[scale=#1]{
\begin{scope}
\begin{scope}  [decoration={border,segment length=#1mm,amplitude=#1mm,angle=-90}]
  \begin{scope}[xshift=0cm]
    \draw (1,0) arc (0:360:1cm and 0.3cm)coordinate [pos=0](e)
    coordinate [pos=0.5](f) coordinate[pos=0.65] (bl)
    coordinate[pos=0.85] (al);
    \node at (0,-0.5) {$#3$};
    \draw[very thin] (-1,0) arc (180:270:1cm and 1.2cm) coordinate[pos=1] (LB);
  \end{scope}
  \begin{scope}[xshift=#2cm]
    \begin{scope}[xshift=2cm]
    \draw (1,0) arc (0:360:1cm and 0.3cm)coordinate [pos=1](f)
    coordinate [pos=0.5](f) coordinate[pos=0.65] (br)
    coordinate[pos=0.85] (ar);
    \draw[very thin] (1,0) arc (0:-90:1cm and 1.2cm) coordinate[pos=1]
    (RB);
  \end{scope}
  \draw[very thin] (LB) -- (RB);
  \draw[very thin] (f) .. controls +(0, -0.4) and +(0, -0.4) .. (e);
  \draw[red, postaction={draw, decorate}] (al) .. controls +(0, -0.6)  and +(0, -0.6) .. (br);
  \draw[red, postaction={draw, decorate}] (ar) .. controls +(0, -1)  and +(0, -1) .. (bl);
\end{scope}
\end{scope}
\end{scope}
}}}

\newcommand{\feMarkedPoint}[1]{\NB{\tikz[scale=#1]{
  \begin{scope}       \draw (0,0) --  +(0.5,0);
   \draw [red, -latex, thick] (0.25, 0) -- ( 0.35,0);
  \end{scope}
}}}

\newcommand{\feOrientation}[1]{\NB{\tikz[scale=#1]{
  \begin{scope}       \draw (0,0) --  +(0.5,0);
   \draw [->] (0.25, 0) -- ( 0.27,0);
  \end{scope}
}}}
%END OF AFTER BEGIN DOC

\section{Introduction}

Sucharit Sarkar in \cite{Sarkar_2020} defined a notion of ribbon distance between two knots and used the $X$-torsion order to give a lower bound on this quantity via Lee Homology. Then Juhász, Miller and Zemke in \cite{juh2020knot} worked on related ideas in the context of knot Floer Homology. In particular, they proved a lower bound on the ribbon distance via knot Floer Homology and furnished examples of ribbon knots whose ribbon distance from the unknot is arbitrarily large. In the knot Floer setting, they also proved a general inequality that relates the torsion order of two knots that are cobordant via a connected cobordism, the number of local maxima and minima, and the genus of the cobordism. They further apply their result towards a lower bound on the bridge index of a knot, as well as numerous other bounds.

On the other hand, before Sarkar's work, Akram Alishahi in \cite{Alishahi_2019} obtained a lower bound on the unknotting number via Bar-Natan (characteristic $2$) Homology. At about the same time Akram Alishahi and Nathan Dowlin in \cite{alishahi2017lee} obtained bounds on the unknotting number via Lee Homology. Interestingly, this lower bound coincides with Sarkar's lower bound on the ribbon distance via Lee Homology. It is thus natural to ask whether there exists a lower bound on the ribbon distance via Bar-Natan Homology, and whether this lower bound agrees with \cite{Alishahi_2019}'s lower bound on the unknotting number from Bar-Natan Homology. This is what we prove in this paper. Additionally, we prove a lower bound on the ribbon distance via $\A$-homology, a link invariant defined recently by Khovanov and Robert in \cite{khovanov2020link}.

We are extremely grateful for the ideas in \cite{Sarkar_2020}. Our proof will mirror \cite{Sarkar_2020}'s ideas, with some modifications to accommodate differences in the skein relations (in particular, the neck-cutting relation). Unlike in Lee Homology, a dot cannot be moved around freely within a connected component of a cobordism (even up to a sign) in Bar-Natan Homology or $\A$-homology. To overcome this, we introduce ``Symmetry Lemmas" (Theorems \ref{mirror} and \ref{mirror1}) that appeal to the symmetry of dotted cobordisms appearing in the argument. Eventually the bound we obtain on the ribbon distance from the unknot via Bar-Natan Homology will coincide with the bound that \cite{Alishahi_2019} achieved on the unknotting number from Bar-Natan Homology. This reinforces the similarity between attaining bounds on the ribbon distance and unknotting number via perturbations of Khovanov Homology, a pattern that was first seen in \cite{Sarkar_2020} for the case of Lee Homology.

The notion of the ribbon distance between two knots was first defined by Sarkar in \cite{Sarkar_2020}. It is defined as follows.

\begin{definition}
Given two knots $K$ and $K'$, the \emph{ribbon distance} between $K$ and $K'$, written $d(K,K')$, is the smallest choice of $k$ for which there exists a sequence of knots $K=K_1,...,K_n=K'$ with there being a ribbon concordance between each $K_i$ and $K_{i+1}$ in any direction with no more than $k$ saddles.
\end{definition}

As discussed in \cite{Sarkar_2020}, the ribbon distance satisfies the triangle inequality and it is finite only when the two knots are concordant to one another.

\subsection{Bar-Natan Homology} Let $F_6$ be the Frobenius system in \cite{khovanov2004link} given by $R_6=\mathbb{F}_2[h]$ and $A_6=R_6[X]/(X^2-hX)$, where $\deg(h)=2$. Here $\mathbb{F}_2$ is taken to be the field $\{0,1\}$, 
%and $A$ is the Frobenius algebra with basis given by $\{1,X\}$ such that:  
and the following holds: $\epsilon(1)=0, \epsilon(X)=1$ and $\Delta(1)=1\otimes X + X\otimes 1 + h1\otimes 1$, $\Delta(X)=X\otimes X$. The link homology theory resulting from $F_6$ is called Bar-Natan's charactertisc $2$ theory or Bar-Natan Homology, and we will abbreviate it to $\BN$. This is a functorial link homology theory. For a link diagram $L$, the Bar-Natan Homology of $L$ will be denoted by $\BN(L)$. In order to describe the lower bound on the ribbon distance of a knot coming from Bar-Natan Homology, we need to define some terminology.

\begin{definition}
\label{defntor}
Let $L$ be a knot. Following \cite{Alishahi_2019}, we declare $\gamma \in \BN(L)$ to be \emph{$h$-torsion} if $h^n \gamma=0$ for some $n>0$. The \emph{order} of $\gamma$ is the smallest choice of $n$ for which $h^n \gamma=0$. 
\end{definition}

\begin{definition}
For a knot $L$, we define the \emph{$h$-torsion order} of the Bar-Natan Homology of $L$, written $\mu(L)$, to be the maximum order of a $h$-torsion homology class in $\BN(L)$. 
\end{definition}

The following are the main results of this paper in the $\BN$-setting. 
\begin{theorem}
\label{side}
If $d$ is the ribbon distance between knots $K$ and $K'$, then $h^d\BN(K)\cong h^d\BN(K')$.
\end{theorem}

\begin{theorem}
\label{bound}
If $d$ is the ribbon distance between knots $K$ and $K'$, then $|\mu(K)-\mu(K')|\leq d$.
\end{theorem}

This is used to give a lower bound on the ribbon distance of a knot from the unknot via Bar-Natan Homology. We note that the $h$-torsion order of the unknot is $0$. Thus Theorem \ref{bound} implies that if $d$ is the ribbon distance between a knot $K$ and the unknot, then $\mu(K)\leq d$. This agrees with the lower bound achieved by \cite[Theorem 1.2]{Alishahi_2019} on the unknotting number of a knot $K$.

\subsection{$\A$-homology} 
\label{Adefn}
$\A$-homology or $\A$-theory is an invariant of links defined recently by Khovanov and Robert in \cite{khovanov2020link} and also discussed in \cite{sano2020fixing}. We briefly recall their definition. Let $R_{\A}=\mathbb{Z}[\alpha_1,\alpha_2]$, with $\deg(\alpha_1)=\deg(\alpha_2)=2$ and $A_{\alpha}=R_{\alpha}[X]/((X-\alpha_1)(X-\alpha_2))$, with $\deg(X)=2$. Comultiplication is given by:

$\Delta(1)=(X-\alpha_1)\otimes 1 + 1\otimes (X-\alpha_2)=(X-\alpha_2)
\otimes 1 + 1\otimes (X-\alpha_1)$,

$\Delta(X-\alpha_1)=(X-\alpha_1)\otimes (X-\alpha_1)$, and
$\Delta(X-\alpha_2)=(X-\alpha_2)\otimes (X-\alpha_2)$.

Realising $X-\alpha_i$ as shifted dots, \cite{khovanov2020link} diagrammatically represents them as $\circled{i}$, i.e. $X_{\circled{i}}:=X-\alpha_i$. Following \cite{khovanov2020link}, we set $X_*=X_{\circled{1}}+X_{\circled{2}}$. This can be described by the relation shown in Figure \ref{skein2}. The link homology theory resulting from the $(R_\alpha, A_\alpha)$ extension is called $\alpha$-homology and the $\A$-homology of a link diagram $L$ is written $\A(L)$. In Section \ref{prelim}, we will show that a $*$ can be moved around freely on a movie of a connected $4$-dimensional cobordism without affecting the induced $\A$-homology map (up to an overall sign). This enables us to define a notion of $X_*$-torsion. 

\begin{figure}
    \centering
    $ \cfDecSquare{0.5}{\star} = \cfDecSquare{0.5}{\circled{1}} + \cfDecSquare{0.5}{\circled{2}}$
    \caption{Pictorial definition of $X_*$; figure taken from \cite[page 9]{khovanov2020link}}
    \label{skein2}
    \end{figure}
    
\begin{definition}
Let $L$ be a knot. Analogous to Definition \ref{defntor}, we declare $\gamma \in \A(L)$ to be \emph{$X_*$-torsion} if $X_*^n \gamma=0$ for some $n>0$. The \emph{order} of $\gamma$ is the smallest choice of $n$ for which $X_*^n \gamma=0$. 
\end{definition}

\begin{definition}
For a knot $L$, we define the \emph{$X_*$-torsion order} of the $\A$-homology of $L$, written $\nu(L)$, to be the maximum order of a $X_*$-torsion homology class in $\A(L)$. 
\end{definition}

Analogous to Theorems \ref{bound} and \ref{side}, the following are the main results of the paper in the $\A$-homology setting.

\begin{theorem}
\label{side1}
If $d$ is the ribbon distance between knots $K$ and $K'$, then $X_*^d\BN(K)\cong \pm  X_*^d\BN(K')$.
\end{theorem}

\begin{theorem}
\label{bound1}
If $d$ is the ribbon distance between knots $K$ and $K'$, then $|\nu(K)-\nu(K')|\leq d$.
\end{theorem}

Table \ref{tab} contains a brief summary of the results of this paper as well as previously known results about ribbon distance and unknotting number bounds from perturbations of Khovanov Homology. 

\begin{table}[!ht]
\setlength\extrarowheight{2pt} % for a bit of visual "breathing space"
\begin{tabularx}{\textwidth}{|C|C|C|C|}
\hline
\textbf{Homology Theory} & \textbf{Brief Description} & \textbf{Quantity that gives the bound} & \textbf{Bound for}  \\\hline
Lee Homology & Frobenius algebra: $R[t][X]/(X^2=t)$; Counit: $1\mapsto 0, X\mapsto 1$;  Comultiplication: $1\mapsto 1\otimes X +X\otimes 1$, $X\mapsto X\otimes X + t1\otimes 1$. & $X$-torsion order & Unknotting number: \cite{alishahi2017lee}, Ribbon distance: \cite{Sarkar_2020} \\
\hline
Bar-Natan's char 2 Homology & Frobenius algebra: $\mathbb{F}_2[h][X]/(X^2-hX)$; $\epsilon(1)=0, \epsilon(X)=1$; $\Delta(1)=1\otimes X + X\otimes 1 + h1\otimes 1$, $\Delta(X)=X\otimes X$. & $h$-torsion order & Unknotting number: \cite{Alishahi_2019}, Ribbon distance: this paper\\
\hline
$\A$-Homology & Frobenius algebra: $\mathbb{Z}[\alpha_1,\alpha_2][X]/((X-\alpha_1)(X-\alpha_2))$;

Comultiplication: 
$\Delta(1)=(X-\alpha_1)\otimes 1 + 1\otimes (X-\alpha_2)=(X-\alpha_2)
\otimes 1 + 1\otimes (X-\alpha_1)$,
$\Delta(X-\alpha_i)=(X-\alpha_i)\otimes (X-\alpha_i)$. & $X_*$-torsion order; $X_*:=2X-(\A_1+\A_2)$ & Ribbon distance: this paper\\
\hline
\end{tabularx}
\caption{Summary of results old and new in different variants of Khovanov Homology. For \cite{Sarkar_2020} and \cite{alishahi2017lee}, $2\neq 0$ in $R$ (for more details about \cite{Sarkar_2020} and \cite{alishahi2017lee} the reader is encouraged to refer to the original papers)}
\label{tab}
\end{table}

\subsection{Acknowledgements}
I am very grateful to Mikhail Khovanov for his suggestions which sparked this project and his helpful discussions on it. I am also very thankful to him for introducing me to $\A$-homology and suggesting the extension to $\A$-homology.

I am also grateful to Sucharit Sarkar for helpful conversations about the project. I would also like to thank him as several ideas in this paper are borrowed from the ideas in his paper \cite{Sarkar_2020}, in which he first defined the notion of ribbon distance and explored a similar question in the setting of Lee Homology. 

I am also thankful to Akram Alishahi for a helpful correspondence about \cite{Alishahi_2019} and for some of the results and ideas from \cite{Alishahi_2019} that are used in this paper. I am also thankful to Orsola Capovilla-Searle for a helpful conversation about attaching points of saddles in a link diagram in the proof of \cite[Corollary 2.3]{Alishahi_2019}. I am also thankful to Sungkyung Kang for a helpful correspondence about link homology theories and ribbon concordances. I am also thankful to Rostislav Akhmechet and Maggie Miller for helpful conversations on related topics.

I am eternally grateful to my advisor Adam Levine, for everything over the past two years which includes (but isn't limited to) introducing me to Khovanov Homology as well as its relationship with ribbon concordances and his many helpful conversations on closely related ideas.

\section{Bounding the ribbon distance via Bar-Natan Homology}
\label{bounds}

The Bar-Natan complex of a link diagram $L$ will be denoted by $C_{\BN}(L)$. The map induced on $\BN$ Homology by a $4$-dimensional cobordism $C$ or its movie $M$ is written $\BN(C)$ or $\BN(M)$ respectively. It is easily verified that $\BN$ Homology satisfies a neck-cutting relation as shown in Figure \ref{neck}, where a solid dot indicates multiplication by $X$, and $h$ indicates multiplication by $h$.

\begin{figure}
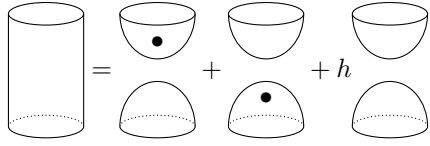

    \centering
    $\cfCircleId{0.5} =
\cfCircleSplit{0.5}{\bullet}{}
+
\cfCircleSplit{0.5}{}{\bullet}
+
h\cfCircleSplit{0.5}{}{}$
    \caption{The neck-cutting relation in $\BN$ Homology; figure modified from \cite[page 9]{khovanov2020link}}
    \label{neck}
    \end{figure}

Just as \cite{LevineZemke} showed that a ribbon concordance induces an injective map on Khovanov Homology, a similar statement holds for $\BN$ Homology as well. 

\begin{theorem}
\label{inj}
If $C:K\to K'$ is a ribbon concordance and $C':K'\to K$ is it opposite cobordism, then $\BN(C' \circ C)=\Id_{\BN(K)}$.
\end{theorem}
\begin{proof}

One way to see this is to follow the approach used by \cite[Proposition 7]{LevineZemke}. Even though the neck-cutting relation in $\BN$ homology  varies slightly from that in Khovanov Homology, Proposition 6, Part 1 of \cite{LevineZemke} continues to hold with ``$\Kh$" replaced by ``$\BN$". Thus a singly dotted, unknotted, unlinked sphere induces the multiplication by $1$ map on $\BN$ homology, while an undotted, unknotted, unlinked sphere induces the $0$ map on $\BN$ homology. Now following the approach of \cite[Proposition 7]{LevineZemke} and applying the appropriate neck-cutting relation in $\BN$ Homology gives the desired result. 
\end{proof}

The following is the main lemma used to prove Theorem \ref{bound}.

\begin{lemma}
\label{main}
Let $C':K'\to K$ be a ribbon concordance of knots with $d$ saddles, and let $C$ be its opposite cobordism. Then 
    $\BN(h^d C'\circ C)=\BN(h^d \Id_K)$
    %\item $h^d\BN(K)$ is isomorphic to $h^d\BN(K')$
\end{lemma}

Lemma \ref{main} will be proved later. First we will state a few facts about $\BN$ Homology that we will need to prove Lemma \ref{main}. Unlike Lee Homology and Khovanov Homology, in $\BN$ Homology when a dot is moved past a crossing in a link diagram, it incurs a ``cost" of $h$, as described by the following theorem:

\begin{theorem} \label{Al2.1} \cite [Lemma 2.1]{Alishahi_2019}
Given a link diagram $L$, and points $p$ and $q$ on either side of a crossing in $L$. Let $X_i$ for $i\in \{p,q\}$ mean the action of a dot at the location $i$ on $L$. Then in $\BN$ homology, $X_p+X_q=h$.
\end{theorem}

Theorem \ref{Al2.1} is used in \cite{Alishahi_2019} to show that in $\BN$ Homology the map induced by a saddle that increases the number of components followed by the saddle in reverse is equivalent to scaling the identity map by $h$:

\begin{theorem} \label{Al2.3} \cite[Corollary 2.3]{Alishahi_2019}
Let $L$ be a link diagram obtained from the link diagram $K$ by performing an oriented saddle that increases the number of components of $K$. Let $f$ denote the $\BN$ map induced by this saddle, and $f'$ be the map induced by doing the saddle in reverse. Then $f'\circ f$ induces a map that is homotopic to multiplication by $h$, i.e. $\BN(f' \circ f)=\BN(h \Id_{K})$
\end{theorem}

$\BN$ homology satisfies a Symmetry Lemma as we will show below:

\begin{lemma} (Symmetry Lemma)
%Suppose that $C$ is a dotted, connected $4$-dimensional cobordism with a vertical plane passing through its center about which the surface underlying $C$ exhibits Symmetry. Then the dotted cobordism $r(C)$ formed by reflecting the location of dots in $C$ induces the same map in $\BN$ homology as the map induced by $C$.
\label{mirror}
Suppose that $C$ is a movie for a dotted, connected $4$-dimensional cobordism. Let $T$ be the same movie as $C$, but without any dots. Suppose that there exists a vertical line passing through the center of $T$ about which the movie $T$ exhibits mirror symmetry. Then the dotted movie $r(C)$ formed by reflecting the location of the dots in $C$ induces the same map in $\BN$ homology as the map induced by $C$.
\end{lemma}

\begin{proof}
Let us assume that the movie $C$ ranges from time, $t = 0$ to $1$, with $T$ exhibiting mirror symmetry about $t=0.5$. First let us work with the case where $C$ has exactly one dot. Suppose that the dot is located at a point $q$ in the frame (link diagram) of $C$ at some time $\tau$, where $\tau < 0.5$. Thus the location of the dot in $C$ is described by the ordered pair $(\tau, q)$, where $\tau$ describes the time of the frame in which the dot is located, and $q$ specifies the location of the dot in that frame. Then in the movie $r(C)$, the dot has been reflected to the same point $q$ but this time located in the frame of $T$ at time $1-\tau$. For convenience, we will refer to the location of the dot in $r(C)$ by the ordered pair $(1-\tau,q)$. 

Because $C$ is a movie for a connected $4$-dimensional cobordism, there exists a way to move the dot from the original location $(\tau, q)$ to $(1-\tau,q)$ using one of two kinds of moves: 

1. making a dot commute with an elementary cobordism (Reidemeister move, birth, death, or saddle) happening away from the dot

2. moving a dot past a crossing in a link diagram in a frame of the movie  

That these two kinds of moves suffice is implied by \cite[Lemma 2.1]{Sarkar_2020}. Applying move 1 does not affect the induced $\BN$ map (this is easy to see). However, move 2 affects the induced $\BN$ map in the following way: it adds $\BN(hT)$ for every application of move 2. 

Now we will show that there exists a way to move the dot from $(\tau, q)$ to $(1-\tau,q)$ such that the number of applications of move 2 is an even number. We note that any way to move the dot from $(\tau, q)$ to $(1-\tau,q)$ must pass a point $q'$ in the $t=0.5$ frame. However, since $T$ is mirror symmetric about the $t=0.5$ frame, every application of moves 1 and 2 that went into moving the dot from $(\tau, q)$ to $(0.5, q')$ can be mirrored when moving the dot from $(0.5, q')$ to $(1- \tau, q)$. Thus the total number of applications of move 2 to take it from $(\tau, q)$ to $(1-\tau,q)$ in this manner is even. Thus the number of times $\BN(hT)$ gets added to the induced map in this process is an even number. So the induced map is unchanged when taking the dot from $(\tau, q)$ to $(1-\tau,q)$. The proof for the case when $C$ has more than $1$ dot follows identically as above, by moving the dots one at a time to their reflection about $t=0.5$.
\end{proof}

\begin{remark}
The above Symmetry Lemma wasn't needed in \cite{Sarkar_2020} since in Lee Homology a dot can move past crossings in a link diagram, without changing the induced map (up to an overall sign).
\end{remark}

Now we are in a position to prove Lemma \ref{main} and Theorem \ref{side}.

\begin{proof} Let $C':K'\to K$ be a ribbon concordance of knots with $d$ saddles, and let $C$ be its opposite cobordism. We will largely mirror the ideas of \cite [Main Theorem: Methods 1 and 2]{Sarkar_2020} with modifications to accommodate differences in the skein relations. A priori $C$ decomposes into a movie that proceeds as follows:

$C_1$) Reidemeister moves and planar isotopies

$C_2$) $d$ saddles

$C_3$) Reidemeister moves and planar isotopies

$C_4$) $d$ deaths 

$C_5$) Reidemeister moves and planar isotopies

However, as \cite{Sarkar_2020} remarks for Lee Homology, due to link invariance of $\BN$, it suffices to assume that $C$ is given by a movie described entirely by $C_4\circ C_3\circ C_2$. Similarly, $C'$ is described by $C_2' \circ C_3' \circ C_4'$, where $C_i'$ denotes doing $C_i$ in the reverse direction. Let $D$ be the cobordism formed by first doing $C$ and then $C'$, but that skips the deaths and the births in between. In other words, $D:=C_2'\circ C_3'\circ C_3\circ C_2$. Starting from $K$ and doing $C_3\circ C_2$ results in the link $K' \bigsqcup U^{d}$, where $U^{d}$ denotes $d$ copies of the unknot. Drawing on a trick from \cite{Sarkar_2020}, we may express $D$ as $C_2'\circ C_3'\circ \Id_{K'\bigsqcup U^{d}} \circ C_3\circ C_2$. 

Now we make use of the neck-cutting relation in $\BN$ Homology shown in Figure \ref{neck}. Even though this differs from the neck-cutting relation in Lee Homology that \cite{Sarkar_2020} appeals to, due to the Symmetry Lemma this will prove to be just as effective as we will now show: we may cut each of the $d$ necks present in the $\Id_{K'\bigsqcup U^{d}}$ cobordism, using the aforementioned neck-cutting relation in $\BN$ homology. Let $R$ be the cobordism formed after cutting all $d$ necks. Then $R$ is a sum of cobordisms, each of which is decorated by dots and/or scaled by powers of $h$. However, just as with \cite{Sarkar_2020}'s neck-cutting, here too the surface $S$ underlying each of the cobordisms in $R$ is the same: $S= C_4'\circ C_4$. We note that $D=C_2'\circ C_3'\circ R\circ C_3\circ C_2$. Thus as $R$ was a sum of cobordisms, each with underlying surface $S$ along with some dots and/or powers of $h$, $D$ too is a sum of cobordisms, each with underlying surface $T:=C_2'\circ C_3'\circ S\circ C_3\circ C_2$ decorated by some dots and/or powers of $h$. Let $\Sigma$ denote the sum of these cobordisms, each with underlying surface $T$. Note that $T$ is a connected surface equal to the cobordism $C'\circ C$, and it has a vertical plane that passes through its center about which it is symmetric. Now, in the sum $\Sigma$, for each dotted cobordism $\lambda$ with some dots and some power of $h$, there exists exactly one other dotted cobordism $\lambda'$ in the sum $\Sigma$ with just as many dots and the same power of $h$ such that the location of the dots is reflected about the vertical plane. All but one of the cobordisms in the sum $\Sigma$ can thus be ``paired" in this fashion. Due to the Symmetry Lemma, the ``paired" cobordisms add up to $0$ in $\Sigma$, and all that's left in $\Sigma$ is the lone undotted cobordism $h^dT$, which is equal to $h^dC_2'\circ C_3'\circ S\circ C_3\circ C_2$. Thus $\BN(D)=\BN(h^d(C'\circ C)$. This can be diagrammatically depicted as shown in Figure \ref{D}, where for the purposes of the diagram, $d$ has been taken to be $1$. 

However, similar to \cite{Sarkar_2020}, we note that $\BN(D)=\BN(C_2'\circ C_2)$, since $C_3'\circ C_3$ is isotopic to the identity cobordism. $C_2'\circ C_2$ first does $d$ saddles, each of which increases the number of components, and then immediately does the $d$ saddles in reverse. Now by applying Theorem \ref{Al2.3} $d$-many times, we conclude that $\BN(D)$ is simply given by $\BN(h^d \Id_{K})$. This can be diagrammatically depicted as shown in Figure \ref{alsoD}, where for the purposes of the diagram, $d$ has been taken to be $1$. Thus $\BN(h^dC'\circ C)=h^d\Id_{\BN(K)}$, completing the proof of Lemma \ref{main}.

Further, due to Theorem \ref{inj}, the image of $\BN(h^d(C'\circ C))$ is isomorphic to $h^d\BN(K')$. However, the image of $\BN(D)$ is equal to that of $\BN(h^d \Id_{K})$, which is equal to $h^d\BN(K)$. Thus $h^d\BN(K')$ is isomorphic to $h^d\BN(K)$, completing the proof of Theorem \ref{side} for the special case when there exists a ribbon concordance between $K$ and $K'$ with at most $d$ saddles. The general case of Theorem \ref{side1} when the ribbon distance between $K$ and $K'$ is $d$, follows trivially from this special case.
\end{proof}

\begin{figure}
    \centering
     \[
    \vcenter{\hbox{
        \begin{tikzpicture}[rotate=90,yscale=0.3]
          \draw[knot] (0,9) circle (1);
          \node[anchor=north] at (-1,9) {$K$};
          
          \draw[knot] (-1,6) arc (-180:0:0.5);
          \draw[knot] (0.2,6) arc (-180:0:0.4);
          \node[anchor=north] at (-1,6) {$\wt{K}$};
          
          \draw[knot] (-1.5,3) arc (-180:0:0.8);
          \draw[knot] (0.8,3) arc (-180:0:0.3);
          \node[anchor=north] at (-1.5,3) {$K'$};
          \node[anchor=north] at (0.8,3) {$U$};

          \draw[knot] (-1.5,0) arc (-180:0:0.8);
          \draw[knot] (0.8,0) arc (-180:0:0.3);
          \node[anchor=north] at (-1.5,0) {$K'$};
          \node[anchor=north] at (0.8,0) {$U$};
          
          \draw[knot] (-1,-3) arc (-180:0:0.5);
          \draw[knot] (0.2,-3) arc (-180:0:0.4);
          \node[anchor=north] at (-1,-3) {$\wt{K}$};
          
          \draw[knot] (-1,-6) arc (-180:0:1);
          \node[anchor=north] at (-1,-6) {$K$};

          \draw (-1,9) -- (-1,6);
          \draw (0,6) to[looseness=10,out=90,in=90] (0.2,6);
          \draw (1,9) -- (1,6);
          \node at (-0.5,7) {$C_2$};
          
          \draw (-1,6) to[looseness=1,out=-90,in=90] (-1.5,3);
          \draw (0,6) to[looseness=1,out=-90,in=90] (0.1,3);
          \draw (0.2,6) to[looseness=1,out=-90,in=90] (0.8,3);
          \draw (1,6) to[looseness=1,out=-90,in=90] (1.4,3);
          \node at (-0.5,4) {$C_3$};
          
          \draw (-1.5,3) -- (-1.5,0);
          \draw (0.1,3) -- (0.1,0);
          \draw (0.8,3) -- (0.8,0);
          \draw (1.4,3) -- (1.4,0);
          \node at (-0.5,1) {$\Id$};

          \draw (-1,-3) to[looseness=1,out=90,in=-90] (-1.5,0);
          \draw (0,-3) to[looseness=1,out=90,in=-90] (0.1,0);
          \draw (0.2,-3) to[looseness=1,out=90,in=-90] (0.8,0);
          \draw (1,-3) to[looseness=1,out=90,in=-90] (1.4,0);
          \node at (-0.5,-2) {$C'_3$};
          
          \draw (-1,-6) -- (-1,-3);
          \draw (0,-3) to[looseness=10,out=-90,in=-90] (0.2,-3);
          \draw (1,-6) -- (1,-3);
          \node at (-0.5,-5) {$C'_2$};
        \end{tikzpicture}}}
    =h
    \vcenter{\hbox{
        \begin{tikzpicture}[rotate=90,yscale=0.3]
          \draw[knot] (0,15) circle (1);
          \node[anchor=north] at (-1,15) {$K$};
          
          \draw[knot] (-1,12) arc (-180:0:0.5);
          \draw[knot] (0.2,12) arc (-180:0:0.4);
          \node[anchor=north] at (-1,12) {$\wt{K}$};
          
          \draw[knot] (-1.5,9) arc (-180:0:0.8);
          \draw[knot] (0.8,9) arc (-180:0:0.3);
          \node[anchor=north] at (-1.5,9) {$K'$};
          \node[anchor=north] at (0.8,9) {$U$};

          \draw[knot] (-1.5,6) arc (-180:0:0.8);
          \node[anchor=north] at (-1.5,6) {$K'$};

          \draw[knot] (-1.5,3) arc (-180:0:0.8);
          \node[anchor=north] at (-1.5,3) {$K'$};

          \draw[knot] (-1.5,0) arc (-180:0:0.8);
          \draw[knot] (0.8,0) arc (-180:0:0.3);
          \node[anchor=north] at (-1.5,0) {$K'$};
          \node[anchor=north] at (0.8,0) {$U$};
          
          \draw[knot] (-1,-3) arc (-180:0:0.5);
          \draw[knot] (0.2,-3) arc (-180:0:0.4);
          \node[anchor=north] at (-1,-3) {$\wt{K}$};
          
          \draw[knot] (-1,-6) arc (-180:0:1);
          \node[anchor=north] at (-1,-6) {$K$};

          \draw (-1,15) -- (-1,12);
          \draw (0,12) to[looseness=10,out=90,in=90] (0.2,12);
          \draw (1,15) -- (1,12);
          \node at (-0.5,13) {$C_2$};
          
          \draw (-1,12) to[looseness=1,out=-90,in=90] (-1.5,9);
          \draw (0,12) to[looseness=1,out=-90,in=90] (0.1,9);
          \draw (0.2,12) to[looseness=1,out=-90,in=90] (0.8,9);
          \draw (1,12) to[looseness=1,out=-90,in=90] (1.4,9);
          \node at (-0.5,10) {$C_3$};
          
          \draw (-1.5,9) -- (-1.5,6);
          \draw (0.1,9) -- (0.1,6);
          \draw (0.8,9) to[looseness=10,out=-90,in=-90] (1.4,9);
          \node at (-0.5,7) {$C_4$};

          \draw (-1.5,6) -- (-1.5,3);
          \draw (0.1,6) -- (0.1,3);
          \node at (-0.5,4) {$\Id_{K'}$};
          \node at (-1,3.6) {};

          \draw (-1.5,3) -- (-1.5,0);
          \draw (0.1,3) -- (0.1,0);
          \draw (0.8,0) to[looseness=10,out=90,in=90] (1.4,0);
          \node at (-0.5,1) {$C'_4$};

          \draw (-1,-3) to[looseness=1,out=90,in=-90] (-1.5,0);
          \draw (0,-3) to[looseness=1,out=90,in=-90] (0.1,0);
          \draw (0.2,-3) to[looseness=1,out=90,in=-90] (0.8,0);
          \draw (1,-3) to[looseness=1,out=90,in=-90] (1.4,0);
          \node at (-0.5,-2) {$C'_3$};
          
          \draw (-1,-6) -- (-1,-3);
          \draw (0,-3) to[looseness=10,out=-90,in=-90] (0.2,-3);
          \draw (1,-6) -- (1,-3);
          \node at (-0.5,-5) {$C'_2$};
        \end{tikzpicture}}}
    \qedhere
    \]    
    \caption{The following equality of cobordisms holds when one passes to $\BN$ Homology, owing to the neck-cutting relation in $\BN$ Homology and the Symmetry Lemma. Figure modified from \cite[page 10]{Sarkar_2020}}
    \label{D}
\end{figure}
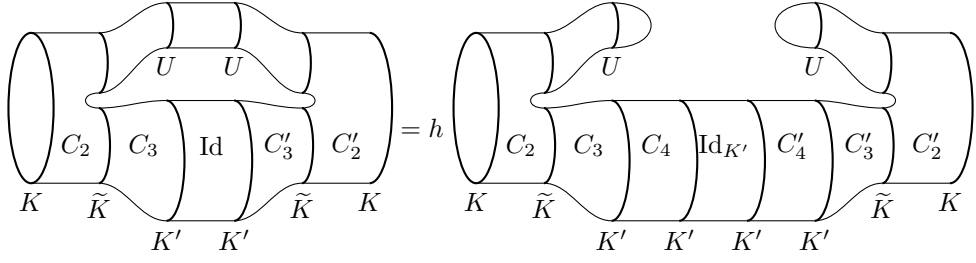

\begin{figure}
    \centering
    \[
    \vcenter{\hbox{
        \begin{tikzpicture}[rotate=90,yscale=0.3]
          \draw[knot] (0,6) circle (1);
          \node[anchor=north] at (-1,6) {$K$};
          
          \draw[knot] (-1,3) arc (-180:0:0.5);
          \draw[knot] (0.2,3) arc (-180:0:0.4);
          \node[anchor=north] at (-1,3) {$\wt{K}$};
          
          \draw[knot] (-1.5,0) arc (-180:0:0.8);
          \draw[knot] (0.8,0) arc (-180:0:0.3);
          \node[anchor=north] at (-1.5,0) {$K'$};
          \node[anchor=north] at (0.8,0) {$U$};
          
          \draw[knot] (-1,-3) arc (-180:0:0.5);
          \draw[knot] (0.2,-3) arc (-180:0:0.4);
          \node[anchor=north] at (-1,-3) {$\wt{K}$};
          
          \draw[knot] (-1,-6) arc (-180:0:1);
          \node[anchor=north] at (-1,-6) {$K$};

          \draw (-1,6) -- (-1,3);
          \draw (0,3) to[looseness=10,out=90,in=90] (0.2,3);
          \draw (1,6) -- (1,3);
          \node at (-0.5,4) {$C_2$};
          
          \draw (-1,3) to[looseness=1,out=-90,in=90] (-1.5,0);
          \draw (0,3) to[looseness=1,out=-90,in=90] (0.1,0);
          \draw (0.2,3) to[looseness=1,out=-90,in=90] (0.8,0);
          \draw (1,3) to[looseness=1,out=-90,in=90] (1.4,0);
          \node at (-0.5,1) {$C_3$};
          
          \draw (-1,-3) to[looseness=1,out=90,in=-90] (-1.5,0);
          \draw (0,-3) to[looseness=1,out=90,in=-90] (0.1,0);
          \draw (0.2,-3) to[looseness=1,out=90,in=-90] (0.8,0);
          \draw (1,-3) to[looseness=1,out=90,in=-90] (1.4,0);
          \node at (-0.5,-2) {$C'_3$};
          
          \draw (-1,-6) -- (-1,-3);
          \draw (0,-3) to[looseness=10,out=-90,in=-90] (0.2,-3);
          \draw (1,-6) -- (1,-3);
          \node at (-0.5,-5) {$C'_2$};
        \end{tikzpicture}}}
    = h\vcenter{\hbox{
        \begin{tikzpicture}[rotate=90,yscale=0.3]
          \draw[knot] (0,3) circle (1);
          \node[anchor=north] at (-1,3) {$K$};
          
          \draw[knot] (-1,0) arc (-180:0:1);
          \node[anchor=north] at (-1,0) {$K$};

          \draw (-1,0) -- (-1,3);
          \draw (1,0) -- (1,3);
          \node at (-0.5,1) {$\Id_{K}$};
          \node at (0,0.6) {};
        \end{tikzpicture}}}
    \]
    \caption{The following equality of cobordisms holds when one passes to $\BN$ Homology, owing to \cite[Corollary 2.3]{Alishahi_2019}. Figure modified from \cite[page 9]{Sarkar_2020}.}
    \label{alsoD}
\end{figure}
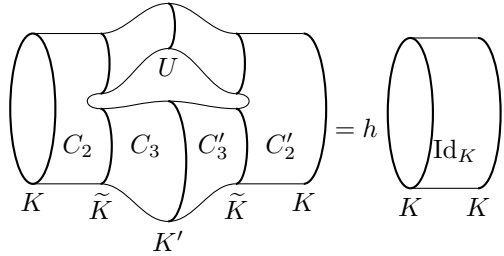

%\subsection{The final step}

\begin{theorem}
\label{chain}
If $d$ is the ribbon distance between two knots that are represented by knot diagrams $M_0$ and $M_n$, then there exist chain maps between the Bar-Natan complexes $u:C_{\BN}(M_0)\to C_{\BN}(M_n)$ and $v:C_{\BN}(M_0)\to C_{\BN}(M_n)$, such that both $u\circ v$ and $v\circ u$ induce multiplication by $h^d$ in $\BN$ Homology.
\end{theorem}

\begin{proof}

Suppose that $d$ is the ribbon distance between the knots $M_0$ and $M_n$. So there exists a sequence of knot diagrams $M_0,...,M_n$ such that $M_i$ is related to $M_{i+1}$ via a movie for a ribbon concordance $C_i:M_i\to M_{i+1}$, or $C_i:M_{i+1}\to M_i$ with at most $d$ saddles. If $C_i$ denotes the movie for a ribbon concordance, its opposite movie will be called $C'_i$. Thus we have a movie $f=D_{n-1}\circ D_{n-2}\circ...\circ D_0$, where $D_j:M_j\to M_{j+1}$ takes on the value $C_j$ or $C'_j$ for each $j$. Similarly, we have a movie $g=E_0 \circ ...\circ E_{n-2}\circ E_{n-1}$, where $E_j:M_{j+1}\to M_{j}$ takes on the value $C_j$ or $C'_j$ for each $j$. Thus $f:M_0\to M_n$ and $g:M_n\to M_0$. Let $f'$ be the movie $h^df$. 

Let $u:C_{\BN}(M_0)\to C_{\BN}(M_n)$ be the chain map induced by $f'$ between the Bar-Natan complexes of $M_0$ and $M_n$ and let $v:C_{\BN}(M_n)\to C_{\BN}(M_0)$ be the chain map induced by $g$ between the Bar-Natan complexes of $M_n$ and $M_0$. Now let $F:\BN(M_0)\to \BN(M_n)$ be the map induced by $f'$ on Bar-Natan Homology. Let $G:\BN(M_n)\to \BN(M_0)$ be the map induced by $g$ on Bar-Natan Homology. We wish to show that $FG$ and $GF$ are multiplication by $h^d$. It suffices to prove the statement for $FG$. We will start from the middle of the movie. That is, we will first consider the part of $FG$ that is induced by the movie $M_{n-1}\to M_{n-1}$ given by $h^d E_{n-1}\circ D_{n-1}$.

\textbf{Case 1:} $D_{n-1}$ is the movie for a ribbon concordance, and $E_{n-1}$ is its opposite movie. In this case, $E_{n-1}\circ D_{n-1}$ induces the identity on $\BN(M_{n-1})$ by Theorem \ref{inj}. Thus, $h^d E_{n-1}\circ D_{n-1}$ induces $h^d \Id_{\BN(M_{n-1})}$.

\textbf{Case 2:} $D_{n-1}$ is the movie for the opposite of a ribbon concordance, and $E_{n-1}$ is the movie for the ribbon concordance. Then by Lemma \ref{main} $h^d E_{n-1}\circ D_{n-1}$ induces $h^d \Id_{\BN(M_{n-1})}$

With this done, we will next show that $h^d E_{n-2}\circ E_{n-1} \circ D_{n-1}\circ D_{n-2}$ induces $h^d \Id_{\BN(M_{n-2})}$. Since we already know that $h^d E_{n-1}\circ D_{n-1}$ induces $h^d \Id_{\BN(M_{n-1})}$, it suffices to show that $h^d E_{n-2}\circ D_{n-2}$ induces $h^d \Id_{\BN(M_{n-2})}$, which once again follows from the identical sort of casework shown above. Thus it is clear that $FG$ is given by $h^d \Id$.
\end{proof}

Now we will need one last fact in order to prove Theorem \ref{bound}. This fact is a special case of \cite[Lemma 3.1]{Alishahi_2019}. 

\begin{lemma}
\label{ext}
Let $L$ and $L'$ be knot diagrams. Suppose that there exist chain maps between the Bar-Natan complexes $u:C_{\BN}(L)\to C_{\BN}(L')$ and $v:C_{\BN}(L')\to C_{\BN}(L)$ such that $u\circ v$ and $v \circ u$ induce multiplication by $h^n$ on Bar-Natan homology. Then $|\mu(L)-\mu(L')|\leq n$.
\end{lemma}
\begin{proof}
This follows as a special case of Lemma 3.1 of \cite{Alishahi_2019}.
\end{proof}

Thus due to Theorem \ref{chain} and Lemma \ref{ext} we have proven Theorem \ref{bound}.

\section{Bounding the ribbon distance via $\A$-homology}
\label{prelim}

The $\A$-homology of a link is as defined in Section \ref{Adefn}. The $(R_{\A},A_{\A})$ extension described in Section \ref{Adefn} is associated with a neck-cutting relation shown in Figure \ref{skein1}. The map induced on $\A$-homology by a $4$-dimensional cobordism $C$ or its movie $M$ is written $\A(C)$ or $\A(M)$ respectively. When necessary, we will denote the the $\A$-complex of a link diagram $L$ by $C_{\A}(L)$. 

\begin{figure}
    \centering
    $\cfCircleId{0.8} =
\cfCircleSplit{0.8}{\circled{1}}{}
+
\cfCircleSplit{0.8}{}{\circled{2}}
=
\cfCircleSplit{0.8}{\circled{2}}{}
+
\cfCircleSplit{0.8}{}{\circled{1}}$ 
    \caption{The neck-cutting relation in $\A$-Homology; figure taken from \cite[page 9]{khovanov2020link}}
    \label{skein1}
    \end{figure}
    
If $C:K_0\to K_1$ is a movie for a $4$-dimensional cobordism between link diagrams $K_0$ and $K_1$, then it is easy to see that placing a $\circled{1}$, or $\circled{2}$, or $*$ at some point of the movie defines a chain map between the $\A$-complexes of $K_0$ and $K_1$. This is because of the fact that when computing the induced $\A$-homology map, the maps induced by $\circled{1}$, or $\circled{2}$, or $*$ commute with the map induced by a saddle. Now we will describe a theorem (alluded to in Section \ref{Adefn}) that states that a $*$ can be moved around freely on a movie of a connected $4$-dimensional cobordism, without affecting the induced $\A$-homology map up to an overall sign.    
    
\begin{theorem}
\label{indiff}
Suppose that $C$ and $D$ are movies for $4$-dimensional cobordisms with the same underlying movie barring the $*$'s on them, such that $C$ and $D$ have the same number of $*$'s on each component, but differ only in the location of $*$'s. Then $\A(C)= \pm \A(D)$.
\end{theorem}    
\begin{proof}
The proof makes use of \cite[Lemma 2.1]{Sarkar_2020}, which essentially states that $C$ and $D$ are related by one of three kinds of movie moves:

1) Adding a $*$ commutes with any other $4$-dimensional elementary cobordism that occurs away from the $*$

2) Adding a $*$ commutes with adding another $*$

3) A $*$ can be moved past a crossing in any link diagram featuring in the movies

It is easy to see why the first two kinds of movie moves do not affect the induced $\A$-map up to an overall sign. That the third movie move does not affect the induced $\A$ map (up to an overall sign) will follow from Theorem \ref{*move} stated later. 
\end{proof}

For a movie $C$ of a connected $4$-dimensional cobordism, we let ``$X_*C$" denote a movie formed by placing a $*$ anywhere along the movie $C$. Then by virtue of Theorem \ref{indiff} $\A(X_*C)$ is well-defined up to an overall sign. Next we remind the reader that just as with other Khovanov perturbations, the map induced on $\A$-homology by a ribbon concordance is injective. 

\begin{theorem}
\label{inj1}
If $C:K\to K'$ is a ribbon concordance and $C':K'\to K$ is it opposite cobordism, then $\A(C' \circ C)=\Id_{\A(K)}$.
\end{theorem}
\begin{proof}

The proof follows nearly identically to the proof for Khovanov Homology shown in \cite[Proposition 7]{LevineZemke}, or the argument outlined for Theorem \ref{inj}.

\end{proof}

The following is the main lemma that will be used to prove Theorem \ref{bound1}.

\begin{lemma}
\label{main1}
Let $C':K'\to K$ be a ribbon concordance of knots with $d$ saddles, and let $C$ be its opposite cobordism. Then 
    $\A(X_*^d C'\circ C)=\pm \A(X_*^d \Id_K)$
    %\item $h^d\BN(K)$ is isomorphic to $h^d\BN(K')$
\end{lemma}

Lemma \ref{main1} will be proved in a bit. First we will state a few facts about $\A$-Homology that we will need to prove Lemma \ref{main1}. 

\begin{theorem}
\label{saddletheorem}
Let $f:K\to L$ be an oriented saddle move performed on a link diagram $K$. Let $f':L\to K$ denote the saddle in reverse done on the link diagram $L$. Let $q$ and $q'$ be the connecting points of $f$ on the link diagram $K$. Then $\A(f' \circ f)= \circled{1}_q + \circled{2}_{q'}=\circled{2}_q + \circled{1}_{q'}$, where $\circled{1}_q$ denotes the map induced by $\A$-homology when $\circled{1}$ is placed at the location $q$ on $K$, and similarly for $\circled{2}_{q'}$.
\end{theorem}
\begin{proof}
The proof follows from the neck-cutting relation in $\A$-homology shown in Figure \ref{skein1}.
\end{proof}

In $\A$-Homology, when a $\circled{1}$ is moved past a crossing in a link diagram it changes to $-\circled{2}$, as described by the following theorem. 

\begin{theorem} \label{moving}
Given a link diagram $L$, and points $p$ and $q$ on either side of a crossing $c$ in $L$. Let $\circled{j}_i$ for $j\in \{1,2\}$ and $i\in \{p,q\}$ mean the map induced on $\A$-homology when a $\circled{j}$ is placed at the point $i$ on $L$. Then in $\A$-homology, $\circled{1}_{q}+\circled{2}_{q'}=0$.
\end{theorem}
\begin{proof}
The proof follows the idea of \cite[Lemma 2.1]{Alishahi_2019} for the most part, except that we need to pay careful attention to signs throughout and account for differences in the skein relations. Without loss of generality, let $c$ be the first crossing of $L$. Let $L_0$ and $L_1$ denote the $0$ and the $1$ resolutions respectively at the crossing $c$ of $L$. Now $L_0$ and $L_1$ can be oriented so that $L_1$ is obtained from $L_0$ by performing an oriented saddle. Let $f$ be the map on $\A$-complexes induced by the this saddle move. So the mapping cone $f:C_\A(L_0)\to C_\A(L_1)$ is $C_\A(L)$ (up to grading shift). Let $f'$ be the map on $\A$-complexes induced by the oriented saddle from $L_1$ to $L_0$. Having written $C_\A(L)$ as the mapping cone, there exists a homotopy $H(a_0,a_1):=(f'(a_1),0)$. Let $\delta$ denote the differential.

Since $c$ is the first crossing of $L$, for the cube of resolutions for $C_\A(L)$ $f$ describes the differential formed by changing the first crossing of $L$. Thus, $\delta H(a_0,a_1) + H\delta (a_0,a_1)= \delta(f'(a_1),0) +H(\delta(a_0),f(a_0)+\delta(a_1))=(\delta (f'(a_1)),f(f'(a_1)))+(f'(f(a_0))+ f'(\delta (a_1)),0)$. However, since $c$ is the first crossing of the link diagram $L$, $\delta (f'(a_1))=-f'(\delta(a_1))$. Thus, $(\delta (f'(a_1)),f(f'(a_1)))+(f'(f(a_0))+ f'(\delta (a_1)),0)=(f'(f(a_0)),f(f'(a_1)))$, which by Theorem \ref{saddletheorem} is equal to $(\circled{1}_p + \circled{2}_q)(a_0,a_1)$.
\end{proof}

\begin{theorem}
Given a link diagram $L$, and points $p$ and $q$ on either side of a crossing $c$ in $L$. In $\A$-homology, the map $\A({X_*}_p L)$ induced by placing a $*$ at the point $p$ on $L$ is equal to $\A(-{X_*}_q)$ induced by placing $-*$ at the point $q$ on $L$.
\end{theorem}
\label{*move}
\begin{proof}
This is a trivial consequence of Theorem \ref{moving} and the fact that $X_*=X_{\circled{1}}+X_{\circled{2}}$.
\end{proof}

Theorem \ref{*move} completes the proof of Theorem \ref{indiff} stated earlier. Theorem \ref{moving} will now be used to show that in $\A$-Homology the map induced by a saddle that increases the number of components followed by the saddle in reverse is equivalent to scaling the identity map by $\pm X_*$:

\begin{theorem} \label{splitsaddle} 
Let $L$ be a link diagram obtained from the link diagram $K$ by performing an oriented saddle that increases the number of components of $K$. Let $f$ denote the $\A$ map induced by this saddle, and $f'$ be the map induced by doing the saddle in reverse. Then $f'\circ f$ induces a map that is homotopic to multiplication by $\pm X_*$, i.e. $\A(f' \circ f)=\pm \A(X_* \Id_{K})$.
\end{theorem}
\begin{proof}
The proof follows the idea of \cite[Corollary 2.3]{Alishahi_2019} in the $\BN$ setting, with a minor change due to differences in the skein relations. Let $p$ and $q$ be the attaching points on the diagram $K$ for the oriented saddle move. Then by Theorem \ref{saddletheorem}, $\A(f' \circ f)$ is given by $\circled{1}_p + \circled{2}_q$. Next we note that $p$ and $q$ lie on the same connected component of the link whose link diagram is $K$. So let $a$ be an arc from $p$ to $q$ on the link diagram $K$. This arc must pass an even number of crossings, so by virtue of Theorem \ref{moving}, $\circled{1}_p=\circled{1}_q$. Thus $\A(f' \circ f)=\circled{1}_p + \circled{2}_q=\circled{1}_q + \circled{2}_q= {X_*}_q$. So by Theorem \ref{*move}, $\A(f' \circ f)=\pm \A(X_*\Id_K)$. 
\end{proof}

Just as we showed for $\BN$ homology in Lemma \ref{mirror}, $\A$-homology also satisfies a Symmetry lemma as we will now show. But in order to describe this statement, it will be useful to define the following:

\begin{definition}
    We say that a movie $M$ for a $4$-dimensional cobordism is \emph{digited}, if it has one or more $\circled{1}$'s or $\circled{2}$'s located on some frame(s) of the movie. In this context, the $\circled{1}$'s or $\circled{2}$'s will be called the \emph{digits} of $M$.
\end{definition}

In what follows, it will often be useful to refer to the digit $\circled{1}$ as the \emph{opposite digit} of $\circled{2}$ and vice versa. 

\begin{lemma} (Symmetry Lemma)
\label{mirror1}
Suppose that $C$ is a singly digited movie for a connected $4$-dimensional cobordism. Let $T$ be the same movie as $C$, but without any digits. Suppose that there exists a vertical line passing through the center of $T$ about which the movie $T$ exhibits mirror symmetry. Then the digited movie $r(C)$ formed by reflecting the location of the digit in $C$ induces the same map in $\A$-homology as the map induced by $C$.
\end{lemma}

\begin{proof}
The proof follows a similar idea to the proof of Theorem \ref{mirror} in the $\BN$ setting. Let us assume that the movie $C$ ranges from time ($t$) $0$ to $1$, with $T$ exhibiting mirror symmetry about $t=0.5$. Suppose that the digit is located at a point $q$ in the frame (link diagram) of $C$ at some time $\tau$, where $\tau < 0.5$. Thus the location of the digit in $C$ is described by the ordered pair $(\tau, q)$, where $\tau$ describes the time of the frame in which the digit is located, and $q$ specifies the location of the digit in that frame. Then in the movie $r(C)$, the digit has been reflected to the same point $q$ but this time located in the frame of $T$ at time $1-\tau$. For convenience, we will refer to the location of the digit in $r(C)$ by the ordered pair $(1-\tau,q)$. 

Because $C$ is a movie for a connected $4$-dimensional cobordism, there exists a way to move the digit from the original location $(\tau, q)$ to $(1-\tau,q)$ using one of two kinds of moves: 

1. making a digit commute with an elementary cobordism happening away from a digit

2. moving a digit past a crossing in a link diagram in a frame of the movie  

That these two kinds of moves suffice is implied by \cite[Lemma 2.1]{Sarkar_2020}. Applying move 1 does not affect the induced $\A$-homology map (this is easy to see). However, move 2 affects the induced $\A$-homology map in the following way: the digit that was moved by virtue of move 2 now switches to the additive inverse of the opposite digit. For example, if $\circled{1}$ was moved past a crossing by virtue of move 2, then it switches to $-\circled{2}$. This is due to Theorem \ref{moving}.

Now we will show that there exists a way to move the digit from $(\tau, q)$ to $(1-\tau,q)$ such that the number of applications of move 2 is an even number. We note that any way to move the digit from $(\tau, q)$ to $(1-\tau,q)$ must pass a point $q'$ in the $t=0.5$ frame. However, since $T$ is mirror symmetric about the $t=0.5$ frame, every application of moves 1 and 2 that went into moving the digit from $(\tau, q)$ to $(0.5, q')$ can be mirrored when moving the digit from $(0.5, q')$ to $(1- \tau, q)$. Thus the total number of applications of move 2 to take it from $(\tau, q)$ to $(1-\tau,q)$ in this manner is an even number. Thus the induced $\A$-homology map is unchanged when taking the digit from $(\tau, q)$ to $(1-\tau,q)$.

\end{proof}

Now we are in a position to prove Lemma \ref{main1} and Theorem \ref{side1}.

\begin{proof} Let $C':K'\to K$ be a ribbon concordance of knots with $d$ saddles, and let $C$ be its opposite cobordism. The proof will follow a similar idea as the proof of Lemma \ref{main}, which in turn is built on ideas of \cite [Main Theorem: Methods 1 and 2]{Sarkar_2020} with modifications to accommodate for differences in the skein relations. A priori $C$ decomposes into a movie that proceeds as follows:

$C_1$) Reidemeister moves and planar isotopies

$C_2$) $d$ saddles

$C_3$) Reidemeister moves and planar isotopies

$C_4$) $d$ deaths 

$C_5$) Reidemeister moves and planar isotopies

However, due to link invariance of $\A$-homology, it suffices to assume that $C$ is given by a movie described entirely by $C_4\circ C_3\circ C_2$. Similarly, $C'$ is described by $C_2' \circ C_3' \circ C_4'$, where $C_i'$ denotes doing $C_i$ in the reverse direction. Let $D$ be the cobordism formed by first doing $C$ and then $C'$, but that skips the deaths and the births in between. In other words, $D:=C_2'\circ C_3'\circ C_3\circ C_2$. Starting from $K$ and doing $C_3\circ C_2$ results in the link $K' \bigsqcup U^{d}$, where $U^{d}$ denotes $d$ copies of the unknot. Drawing on a trick from \cite{Sarkar_2020}, we may express $D$ as $C_2'\circ C_3'\circ \Id_{K'\bigsqcup U^{d}} \circ C_3\circ C_2$. 

Now we make use of the neck-cutting relation in $\A$-Homology shown in Figure \ref{skein1}. Even though this differs from the neck-cutting relation in Lee Homology that \cite{Sarkar_2020} appeals to, due to the Symmetry Lemma this will prove to be just as effective as we will now show: order the $d$ necks present in the $\Id_{K'\bigsqcup U^{d}}$ cobordism from $1$ to $d$. Let $R_i$ denote the cobordism formed after cutting the first $i$ necks present in the $\Id_{K'\bigsqcup U^{d}}$ cobordism using the neck-cutting relation shown in Figure \ref{skein1}. Let $S_i:=C_2'\circ C_3'\circ R_i\circ C_3\circ C_2$. Thus, $\A(S_d)= \A (D)$. Now note that $S_i$ is a sum of digited cobordisms each with the same underlying surface. Let $T_i$ denote this common surface underlying the cobordisms in the sum $S_i$. Thus, $T_d=C'\circ C$.

$R_1$ is the cobordism formed after cutting the $1$st neck. So $S_1=C_2'\circ C_3'\circ R_1\circ C_3\circ C_2$ is a sum of two digited cobordisms $U_1$ and $V_1$, such that the surface $T_1$ underlying both $U_1$ and $V_1$ is the same. However, the digits appearing in $U_1$ and $V_1$ are opposites of each other, and they are mirror reflected about a vertical plane passing through the center of the underlying surface $T_1$. So by the Symmetry Lemma, $\A(U_1+V_1)=\pm \A(X_*T_1)$. Thus, $\A(S_1)=\pm \A(X_*T_1)$. Then we can cut the $2$nd neck and an identical argument shows that $\A(S_2)=\pm \A(X_*^2 T_2)$. It follows that $\A(S_i)=\pm \A(X_*^{i}T_i)$ for all $i\in \{1,...,d\}$. Thus, $\A(D)=\A(S_d)=\pm \A(X_*^d T_d)=\pm \A(X_*^d C' \circ C)$. This can be diagrammatically depicted as shown in Figure \ref{D1}, where for the purposes of the diagram, $d$ has been taken to be $1$.

However, note that $\A(D)=\A(C_2'\circ C_2)$, since $C_3'\circ C_3$ is isotopic to the identity cobordism. $C_2'\circ C_2$ first does $d$ saddles, each of which increases the number of components, and then immediately does the $d$ saddles in reverse. Now by applying Theorem \ref{splitsaddle} $d$-many times, we conclude that $\A(D)$ is simply given by $\pm \A(X_*^d \Id_{K})$. This can be diagrammatically depicted as shown in Figure \ref{alsoD1}, where for the purposes of the diagram, $d$ has been taken to be $1$. Thus $\A(X_*^d C'\circ C)=\pm X_*^d\Id_{\A(K)}$, completing the proof of Lemma \ref{main1}.

Further, due to Theorem \ref{inj1}, the image of $\A(X_*^d(C'\circ C))$ is isomorphic to $\pm X_*^d\A(K')$. However, the image of $\A(D)$ is equal to that of $\A(X_*^d \Id_{K})$, which is equal to $X_*^d\A(K)$. Thus $X_*^d\A(K')$ is isomorphic to $\pm X_*^d\A(K)$, completing the proof of Theorem \ref{side1} for the special case when there exists a ribbon concordance between $K$ and $K'$ with at most $d$ saddles. The general case of Theorem \ref{side1} when the ribbon distance between $K$ and $K'$ is $d$, follows trivially from this special case. 
\end{proof}

\begin{figure}
    \centering
    \[
    \vcenter{\hbox{
        \begin{tikzpicture}[rotate=90,yscale=0.3]
          \draw[knot] (0,9) circle (1);
          \node[anchor=north] at (-1,9) {$K$};
          
          \draw[knot] (-1,6) arc (-180:0:0.5);
          \draw[knot] (0.2,6) arc (-180:0:0.4);
          \node[anchor=north] at (-1,6) {$\wt{K}$};
          
          \draw[knot] (-1.5,3) arc (-180:0:0.8);
          \draw[knot] (0.8,3) arc (-180:0:0.3);
          \node[anchor=north] at (-1.5,3) {$K'$};
          \node[anchor=north] at (0.8,3) {$U$};

          \draw[knot] (-1.5,0) arc (-180:0:0.8);
          \draw[knot] (0.8,0) arc (-180:0:0.3);
          \node[anchor=north] at (-1.5,0) {$K'$};
          \node[anchor=north] at (0.8,0) {$U$};
          
          \draw[knot] (-1,-3) arc (-180:0:0.5);
          \draw[knot] (0.2,-3) arc (-180:0:0.4);
          \node[anchor=north] at (-1,-3) {$\wt{K}$};
          
          \draw[knot] (-1,-6) arc (-180:0:1);
          \node[anchor=north] at (-1,-6) {$K$};

          \draw (-1,9) -- (-1,6);
          \draw (0,6) to[looseness=10,out=90,in=90] (0.2,6);
          \draw (1,9) -- (1,6);
          \node at (-0.5,7) {$C_2$};
          
          \draw (-1,6) to[looseness=1,out=-90,in=90] (-1.5,3);
          \draw (0,6) to[looseness=1,out=-90,in=90] (0.1,3);
          \draw (0.2,6) to[looseness=1,out=-90,in=90] (0.8,3);
          \draw (1,6) to[looseness=1,out=-90,in=90] (1.4,3);
          \node at (-0.5,4) {$C_3$};
          
          \draw (-1.5,3) -- (-1.5,0);
          \draw (0.1,3) -- (0.1,0);
          \draw (0.8,3) -- (0.8,0);
          \draw (1.4,3) -- (1.4,0);
          \node at (-0.5,1) {$\Id$};

          \draw (-1,-3) to[looseness=1,out=90,in=-90] (-1.5,0);
          \draw (0,-3) to[looseness=1,out=90,in=-90] (0.1,0);
          \draw (0.2,-3) to[looseness=1,out=90,in=-90] (0.8,0);
          \draw (1,-3) to[looseness=1,out=90,in=-90] (1.4,0);
          \node at (-0.5,-2) {$C'_3$};
          
          \draw (-1,-6) -- (-1,-3);
          \draw (0,-3) to[looseness=10,out=-90,in=-90] (0.2,-3);
          \draw (1,-6) -- (1,-3);
          \node at (-0.5,-5) {$C'_2$};
        \end{tikzpicture}}}
    =\pm 
    \vcenter{\hbox{
        \begin{tikzpicture}[rotate=90,yscale=0.3]
          \draw[knot] (0,15) circle (1);
          \node[anchor=north] at (-1,15) {$K$};
          
          \draw[knot] (-1,12) arc (-180:0:0.5);
          \draw[knot] (0.2,12) arc (-180:0:0.4);
          \node[anchor=north] at (-1,12) {$\wt{K}$};
          
          \draw[knot] (-1.5,9) arc (-180:0:0.8);
          \draw[knot] (0.8,9) arc (-180:0:0.3);
          \node[anchor=north] at (-1.5,9) {$K'$};
          \node[anchor=north] at (0.8,9) {$U$};

          \draw[knot] (-1.5,6) arc (-180:0:0.8);
          \node[anchor=north] at (-1.5,6) {$K'$};

          \draw[knot] (-1.5,3) arc (-180:0:0.8);
          \node[anchor=north] at (-1.5,3) {$K'$};

          \draw[knot] (-1.5,0) arc (-180:0:0.8);
          \draw[knot] (0.8,0) arc (-180:0:0.3);
          \node[anchor=north] at (-1.5,0) {$K'$};
          \node[anchor=north] at (0.8,0) {$U$};
          
          \draw[knot] (-1,-3) arc (-180:0:0.5);
          \draw[knot] (0.2,-3) arc (-180:0:0.4);
          \node[anchor=north] at (-1,-3) {$\wt{K}$};
          
          \draw[knot] (-1,-6) arc (-180:0:1);
          \node[anchor=north] at (-1,-6) {$K$};

          \draw (-1,15) -- (-1,12);
          \draw (0,12) to[looseness=10,out=90,in=90] (0.2,12);
          \draw (1,15) -- (1,12);
          \node at (-0.5,13) {$C_2$};
          
          \draw (-1,12) to[looseness=1,out=-90,in=90] (-1.5,9);
          \draw (0,12) to[looseness=1,out=-90,in=90] (0.1,9);
          \draw (0.2,12) to[looseness=1,out=-90,in=90] (0.8,9);
          \draw (1,12) to[looseness=1,out=-90,in=90] (1.4,9);
          \node at (-0.5,10) {$C_3$};
          
          \draw (-1.5,9) -- (-1.5,6);
          \draw (0.1,9) -- (0.1,6);
          \draw (0.8,9) to[looseness=10,out=-90,in=-90] (1.4,9);
          \node at (-0.5,7) {$C_4$};

          \draw (-1.5,6) -- (-1.5,3);
          \draw (0.1,6) -- (0.1,3);
          \node at (-0.5,4) {$\Id_{K'}$};
          \node at (-1,3.6) {*};

          \draw (-1.5,3) -- (-1.5,0);
          \draw (0.1,3) -- (0.1,0);
          \draw (0.8,0) to[looseness=10,out=90,in=90] (1.4,0);
          \node at (-0.5,1) {$C'_4$};

          \draw (-1,-3) to[looseness=1,out=90,in=-90] (-1.5,0);
          \draw (0,-3) to[looseness=1,out=90,in=-90] (0.1,0);
          \draw (0.2,-3) to[looseness=1,out=90,in=-90] (0.8,0);
          \draw (1,-3) to[looseness=1,out=90,in=-90] (1.4,0);
          \node at (-0.5,-2) {$C'_3$};
          
          \draw (-1,-6) -- (-1,-3);
          \draw (0,-3) to[looseness=10,out=-90,in=-90] (0.2,-3);
          \draw (1,-6) -- (1,-3);
          \node at (-0.5,-5) {$C'_2$};
        \end{tikzpicture}}}
    \qedhere
    \]    
    \caption{The following equality of cobordisms holds when one passes to $\A$-Homology, owing to the neck-cutting relation in $\A$-Homology and the Symmetry Lemma in $\A$-homology. Figure modified from \cite[page 10]{Sarkar_2020}}
    \label{D1}
\end{figure}

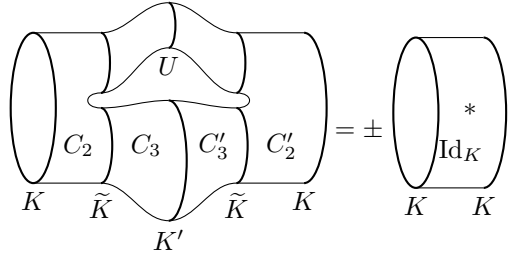
\begin{figure}
    \centering
     \[
    \vcenter{\hbox{
        \begin{tikzpicture}[rotate=90,yscale=0.3]
          \draw[knot] (0,6) circle (1);
          \node[anchor=north] at (-1,6) {$K$};
          
          \draw[knot] (-1,3) arc (-180:0:0.5);
          \draw[knot] (0.2,3) arc (-180:0:0.4);
          \node[anchor=north] at (-1,3) {$\wt{K}$};
          
          \draw[knot] (-1.5,0) arc (-180:0:0.8);
          \draw[knot] (0.8,0) arc (-180:0:0.3);
          \node[anchor=north] at (-1.5,0) {$K'$};
          \node[anchor=north] at (0.8,0) {$U$};
          
          \draw[knot] (-1,-3) arc (-180:0:0.5);
          \draw[knot] (0.2,-3) arc (-180:0:0.4);
          \node[anchor=north] at (-1,-3) {$\wt{K}$};
          
          \draw[knot] (-1,-6) arc (-180:0:1);
          \node[anchor=north] at (-1,-6) {$K$};

          \draw (-1,6) -- (-1,3);
          \draw (0,3) to[looseness=10,out=90,in=90] (0.2,3);
          \draw (1,6) -- (1,3);
          \node at (-0.5,4) {$C_2$};
          
          \draw (-1,3) to[looseness=1,out=-90,in=90] (-1.5,0);
          \draw (0,3) to[looseness=1,out=-90,in=90] (0.1,0);
          \draw (0.2,3) to[looseness=1,out=-90,in=90] (0.8,0);
          \draw (1,3) to[looseness=1,out=-90,in=90] (1.4,0);
          \node at (-0.5,1) {$C_3$};
          
          \draw (-1,-3) to[looseness=1,out=90,in=-90] (-1.5,0);
          \draw (0,-3) to[looseness=1,out=90,in=-90] (0.1,0);
          \draw (0.2,-3) to[looseness=1,out=90,in=-90] (0.8,0);
          \draw (1,-3) to[looseness=1,out=90,in=-90] (1.4,0);
          \node at (-0.5,-2) {$C'_3$};
          
          \draw (-1,-6) -- (-1,-3);
          \draw (0,-3) to[looseness=10,out=-90,in=-90] (0.2,-3);
          \draw (1,-6) -- (1,-3);
          \node at (-0.5,-5) {$C'_2$};
        \end{tikzpicture}}}
    = \pm \vcenter{\hbox{
        \begin{tikzpicture}[rotate=90,yscale=0.3]
          \draw[knot] (0,3) circle (1);
          \node[anchor=north] at (-1,3) {$K$};
          
          \draw[knot] (-1,0) arc (-180:0:1);
          \node[anchor=north] at (-1,0) {$K$};

          \draw (-1,0) -- (-1,3);
          \draw (1,0) -- (1,3);
          \node at (-0.5,1) {$\Id_{K}$};
          \node at (0,0.6) {*};
        \end{tikzpicture}}}
    \]
    \caption{The following equality of cobordisms holds when one passes to $\A$-Homology, owing to Theorem \ref{splitsaddle}. Figure modified from \cite[page 9]{Sarkar_2020}.}
    \label{alsoD1}
\end{figure}

%\subsection{The final step}

\begin{theorem}
\label{chain1}
If $d$ is the ribbon distance between two knots that are represented by knot diagrams $M_0$ and $M_n$, then there exist chain maps between the $\A$-complexes $u:C_{\A}(M_0)\to C_{\A}(M_n)$ and $v:C_{\A}(M_0)\to C_{\A}(M_n)$, such that both $u\circ v$ and $v\circ u$ induce multiplication by $\pm X_*^d$ in $\A$-Homology.
\end{theorem}

\begin{proof}

Suppose that $d$ is the ribbon distance between the knots represented by $M_0$ and $M_n$. So there exists a sequence of knot diagrams $M_0,...,M_n$ such that $M_i$ is related to $M_{i+1}$ via a movie for a ribbon concordance $C_i:M_i\to M_{i+1}$, or $C_i:M_{i+1}\to M_i$ with at most $d$ saddles. If $C_i$ denotes the movie for a ribbon concordance, its opposite movie will be called $C'_i$. Thus we have a movie $f=D_{n-1}\circ D_{n-2}\circ...\circ D_0$, where $D_j:M_j\to M_{j+1}$ takes on the value $C_j$ or $C'_j$ for each $j$. Similarly, we have a movie $g=E_0 \circ ...\circ E_{n-2}\circ E_{n-1}$, where $E_j:M_{j+1}\to M_{j}$ takes on the value $C_j$ or $C'_j$ for each $j$. Thus $f:M_0\to M_n$ and $g:M_n\to M_0$. Let $f'$ be the movie $X_*^d f$ for some choice of location for the $*$'s anywhere in the movie $f$. 

Let $u:C_{\A}(M_0)\to C_{\A}(M_n)$ be the chain map induced by $f'$ between the $\A$-complexes of $M_0$ and $M_n$ and let $v:C_{\A}(M_n)\to C_{\A}(M_0)$ be the chain map induced by $g$ between the $\A$-complexes of $M_n$ and $M_0$. Now let $F:\A(M_0)\to \A(M_n)$ be the map induced by $f'$ on $\A$-Homology. Let $G:\A(M_n)\to \A(M_0)$ be the map induced by $g$ on $\A$-Homology. We wish to show that $FG$ and $GF$ are multiplication by $\pm X_*^d$. It suffices to prove the statement for $FG$. We will start from the middle of the movie. That is, we will first consider the part of $FG$ that is induced by the movie $M_{n-1}\to M_{n-1}$ given by $X_*^d E_{n-1}\circ D_{n-1}$, for some choice of location for the $*$'s.

\textbf{Case 1:} $D_{n-1}$ is the movie for a ribbon concordance, and $E_{n-1}$ is its opposite movie. In this case, $E_{n-1}\circ D_{n-1}$ induces the identity on $\A(M_{n-1})$ by Theorem \ref{inj1}. Thus, $X_*^d E_{n-1}\circ D_{n-1}$ induces $\pm X_*^d \Id_{\A(M_{n-1})}$.

\textbf{Case 2:} $D_{n-1}$ is the opposite of a movie for a ribbon concordance, and $E_{n-1}$ is the movie for a ribbon concordance. Then by Lemma \ref{main1} $X_*^d E_{n-1}\circ D_{n-1}$ induces $\pm X_*^d \Id_{\A(M_{n-1})}$ 

With this done, we will next show that $\pm X_*^d E_{n-2}\circ E_{n-1} \circ D_{n-1}\circ D_{n-2}$ induces $\pm X_*^d \Id_{\A(M_{n-2})}$. Since we already know that $X_*^d E_{n-1}\circ D_{n-1}$ induces $\pm X_*^d \Id_{\A(M_{n-1})}$, it suffices to show that $X_*^d E_{n-2}\circ D_{n-2}$ induces $\pm X_*^d \Id_{\A(M_{n-2})}$, which once again follows from the identical sort of casework shown above. Thus it is clear that $FG$ is given by $\pm X_*^d \Id$, as desired.
\end{proof}

Now we will need one last fact in order to prove Theorem \ref{bound1}. This is analogous to Lemma \ref{ext} from the $\BN$ setting, along with an additional connectedness condition.

\begin{lemma}
\label{ext1}
Let $L$ and $L'$ be knot diagrams. Suppose that there exist chain maps between the $\A$-complexes $u:C_{\A}(L)\to C_{\A}(L')$ and $v:C_{\A}(L')\to C_{\A}(L)$ such that: 

1. (Connectedness condition) $u$ and $v$ are induced by movies (potentially with some $*$'s) representing connected $4$-dimensional cobordisms (potentially with some $*$'s), and

2. $u\circ v$ and $v \circ u$ induce multiplication by $\pm X_*^n$ on $\A$-homology. 

Then $|\nu(L)-\nu(L')|\leq n$.
\end{lemma}
\begin{proof}
This argument essentially follows the ideas of \cite[Lemma 3.1]{Alishahi_2019} along with an additional connectedness condition. Let $U$ denote the map induced by $u$ on $\A$-homology, and $V$ denote the map induced by $v$ on $\A$-homology. Suppose that $\gamma \in \A(L)$ is $X_*$-torsion. Then $U(\gamma) \in \A(L')$ is also $X_*$-torsion because $U$ is induced by a movie for a connected cobordism, which (by Theorem \ref{indiff}) allows $*$'s to move freely on this movie without affecting the induced $\A$-homology map up to an overall sign. Once again making use of the connectedness condition (and Theorem \ref{indiff}) and the fact that $V\circ U=\pm X_*^n$, we get that order of $X_*^n \gamma \leq $ the order of $U(\gamma)$. Similarly, the order of $U(\gamma) \leq $ the order of $\gamma$ using the connectedness condition (and Theorem \ref{indiff}). From this it follows that the order of $\gamma \leq$ the order of $U(\gamma)$ $+ n$. Thus, $\nu(L)\leq \nu(L') + n$. In this fashion it also follows that $\nu(L')\leq \nu(L)+n$. 
\end{proof}

Thus due to Theorem \ref{chain1} and Lemma \ref{ext} we have proven Theorem \ref{bound1}.

\bibliographystyle{amsalpha}
\bibliography{bibliography}

\end{document}